%% file: cables.tex
\documentclass[11pt]{amsart}

\input{macros.tex}

\title[Cabled knots, Dehn surgery, and left-orderability]{On cabled knots, Dehn surgery, and left-orderable fundamental groups}
\date{March 11, 2011}

\author[Adam Clay]{Adam Clay}
\address{
CIRGET, 
Universit\'e du Qu\'ebec \`a Montr\'eal, 
Case postale 8888, Succursale centre-ville, 
Montr\'eal QC, H3C 3P8.}
\email{aclay@cirget.ca}

\author[Liam Watson]{Liam Watson}
\thanks{Both authors partially supported by NSERC postdoctoral fellowships}
\address{Department of Mathematics, UCLA, 520 Portola Plaza, Los Angeles CA, 90095.}
\email{lwatson@math.ucla.edu}

\begin{document}

\begin{abstract}
Previous work of the authors establishes a criterion on the fundamental group of a knot complement that determines when Dehn surgery on the knot will have a fundamental group that is not left-orderable \cite{CW2010}. We provide a refinement of this criterion by introducing the notion of a {\em decayed} knot; it is shown that Dehn surgery on decayed knots produces surgery manifolds that have non-left-orderable fundamental group for all sufficiently positive surgeries. As an application, we prove that sufficiently positive cables of decayed knots are always decayed knots.  These results mirror properties of L-space surgeries in the context of Heegaard Floer homology.
\end{abstract}

\maketitle

\input{section/introduction}

\section{A practical reformulation of Theorem \ref{thm:decay}}\label{sec:reformulation}

\input{section/reformulation2}

\section{The proof of Theorem \ref{thm:cable}}\label{sec:cable}

\input{section/proof-option}

\section{Surgery on satellites}\label{sec:satellite}

\input{section/satellite-alt}

\bibliographystyle{plain}

\bibliography{cables}

\end{document}

%% file: macros.tex
 \usepackage{amsmath,amsthm,amsfonts,amssymb,verbatim}
 \usepackage{url}
 \usepackage{graphicx}
 \usepackage[all]{xy}
 \usepackage{wrapfig}
 \usepackage{pinlabel}
 \usepackage{subfigure}
 
\def\co{\colon\thinspace}

\newcommand{\sP}{\mathcal{P}}
\newcommand{\sS}{\mathcal{S}}
\newcommand{\sQ}{\mathcal{Q}}

\newcommand{\bN}{\mathbb{N}}
\newcommand{\bZ}{\mathbb{Z}}
\newcommand{\bQ}{\mathbb{Q}}

\newcommand{\pq}{\frac{p}{q}}

\newcommand{\mC}{\mu_C}
\newcommand{\lC}{\lambda_C}

\newcommand{\HFhat}{\widehat{\operatorname{HF}}}
\newcommand{\rk}{\operatorname{rk}}

\newtheorem{theorem}{Theorem}

\newtheorem{definition}[theorem]{Definition}
\newtheorem{question}[theorem]{Question}

\newtheorem{proposition}[theorem]{Proposition}
\newtheorem{remark}[theorem]{Remark}
\newtheorem{lemma}[theorem]{Lemma}

%% file: section/introduction.tex

\begin{definition}
A group $G$ is left-orderable if there exists a partition of the group elements \[G= \sP\sqcup\{1\}\sqcup\sP^{-1}\] satisfying $\sP\cdot\sP\subseteq\sP$ and  $\sP\ne\emptyset$. The subset $\sP$ is called a positive cone. 
\end{definition}

This is equivalent to $G$ admitting a left-invariant strict total ordering.  For background on left-orderable groups relevant to this paper see \cite{BRW2005,CW2010}; a standard reference for the theory of left-orderable groups is \cite{KME96}.  As established by Boyer, Rolfsen and Wiest \cite{BRW2005} (compare \cite{HS1985}), the fundamental group $\pi_1(K)$ of the complement of a knot $K$ in $S^3$ is always left-orderable. Indeed, this follows from the fact that any compact, connected, irreducible, orientable $3$-manifold with positive first Betti number has left-orderable fundamental group \cite[Theorem 1.1]{BRW2005}.  However, the question of left-orderability for fundamental groups of rational homology $3$-spheres is considerably more subtle (see \cite{BRW2005,CW2010}) and seems closely tied to certain codimension one structures on the $3$-manifold (see \cite{BRW2005, CD2003, RSS2003}). Continuing along the lines of \cite{CW2010} this paper focuses on Dehn surgery, an operation on knots that produces rational homology $3$-spheres.   We recall this construction in order to fix notation and conventions. 

For any knot $K$ in $S^3$ there is a preferred generating set for the peripheral subgroup $\bZ\oplus\bZ\subset\pi_1(K)$ provided by the knot meridian $\mu$ and the Seifert longitude $\lambda$. The latter is uniquely determined (up to orientation) by the existence of a Seifert surface for $K$. We orient $\mu$ so that it links positively with $K$, and orient $\lambda$ so that $\mu\cdot\lambda=1$. For any  rational number $r$ with reduced form $\pq$ we denote the peripheral element $\mu^p\lambda^q$ by $\alpha_r$. At the level of the fundamental group, the result of Dehn surgery along $\alpha_r$ is summarized by the short exact sequence \[1\to\langle\langle \alpha_r\rangle\rangle\to \pi_1(K) \to \pi_1(S^3_r(K)) \to 1.\] Here $\langle\langle \alpha_r\rangle\rangle$ denotes the normal closure of $\alpha_r$, and $S^3_r(K)$ is the $3$-manifold obtained by attaching a solid torus to the boundary of $S^3 \smallsetminus \nu(K)$, sending the meridian of the torus to a simple closed curve representing the class 
\[[\alpha_r]\in H_1(\partial(S^3\smallsetminus\nu(K));\bZ)/\{\pm1\}.\]
 We will blur the distinction between $\alpha_r$ as an element of the fundamental group or as a primitive class in the (projective) first homology of the boundary, and refer to these peripheral elements as slopes. 

While many examples of rational homology 3-spheres have left-orderable fundamental group \cite{BRW2005}, there exist infinite families of knots for which sufficiently positive Dehn surgery (that is, along a slope parametrized by a suitable large rational number) yields a manifold with non-left-orderable fundamental group  \cite{CW2010}. To make this precise, consider the set of slopes \[\sS_r = \{ \alpha_{r'} | r' \ge r \}\] for some fixed rational $r$. 
\begin{definition}
A nontrivial knot $K$ in $S^3$ is called $r$-decayed if, for any positive cone $\sP$ in $\pi_1(K)$, either   $\sP\cap\sS_r=\sS_r$ or $\sP\cap\sS_r=\emptyset$. 
\end{definition}

The existence of decayed knots is established in \cite{CW2010}. For example, the torus knot $T_{p,q}$ is $(pq-1)$-decayed (for $p,q>0$), and the $(-2,3,q)$-pretzel knot is $(10+q)$-decayed for odd $q\ge5$ (see Theorem \ref{thm:decayedknots}). Our interest in this property stems from the following:

\begin{theorem}\label{thm:decay}
If $K$ is $r$-decayed then $\pi_1(S^3_{r'}(K))$ is not left-orderable for all $r'\ge r$.
\end{theorem}

As a result, it is not restrictive to assume that $r$ is a positive rational number since $\pi_1(S^3_0(K))$ is always left-orderable \cite{BRW2005}. Notice however that it is not immediately clear how Theorem  \ref{thm:decay} might be applied in practice, as there is no obvious method for checking when a knot is $r$-decayed.  For this reason, in Section \ref{sec:reformulation} we describe an equivalent formulation of $r$-decay whose statement is more technical, but easier to verify, than the definition.  Together with the proof of Theorem \ref{thm:decay}, the results of Section \ref{sec:reformulation} provide a useful refinement of the ideas from \cite{CW2010}.

Results connecting left-orderability and Dehn surgery may be expected to mirror similar results relating to L-spaces,
since there is no known example of an L-space with left-orderable fundamental group, while many L-spaces have fundamental group that is not left-orderable.
(see \cite{BGW2010,CR2010,CW2010,Peters2009,Watson2009}).   Recall that an L-space is a rational homology sphere with Heegaard Floer homology that is as simple as possible, in the sense that $\rk\HFhat(Y)=|H_1(Y;\bZ)|$ (see \cite{OSz2005-lens}).  Theorem \ref{thm:decay} mirrors a fundamental property of knots admitting L-space surgeries: if $S_n^3(K)$ is an L-space, then $S^3_r(K)$ is an L-space as well for any $r\ge n$.  


In the interest of further investigating left-orderability of fundamental groups of 3-manifolds along the lines of \cite{CW2010}, we consider the behaviour of Dehn surgery on cables of $r$-decayed knots (for necessary background, see Section \ref{sec:cable}). Denoting the $(p,q)$-cable of the knot $K$ as $C_{p,q}(K)$, the main theorem of this article is:




\begin{theorem}\label{thm:cable}
If K is $r$-decayed then $C_{p,q}(K)$ is $pq$-decayed whenever $\frac{q}{p}>r$. 
\end{theorem}

The proof of Theorem \ref{thm:cable} is contained in Section \ref{sec:cable}. Notice that combining Theorem \ref{thm:decayedknots} and Theorem \ref{thm:cable} provides a rather large class of knots for which sufficiently positive surgery yields a non-left-orderable fundamental group. 

Dehn surgery on cabled knots and non-left-orderability of the resulting fundamental groups may again be viewed in the context of Heegaard Floer homology.  Referring to knots admitting L-space surgeries as L-space knots, Hedden proves: 
\begin{theorem}\cite[Theorem 1.10]{Hedden2009}
If $K$ is an L-space knot  then $C_{p,q}(K)$ is an L-space knot whenever $\frac{q}{p}\ge2g(K)-1$. 
\end{theorem}
Here, the quantity $g(K)$ is the Seifert genus of $K$. Note that the converse of this statement has been recently established by Hom \cite{Hom2011}. 

In order to assess the strength of Theorem \ref{thm:cable}, it is natural to ask when Dehn surgery on a cable knot yields a manifold that has left-orderable fundamental group. It turns out that, in the case that $K$ is $r$-decayed, Theorem \ref{thm:cable} is close to describing all possible non-left-orderable surgeries on a cable knot $C_{p,q}(K)$, in the following sense:  
\begin{theorem}
\label{thm:locable}
Suppose that $C$ is the $(p,q)$-cable of some knot. If $r \in \bQ$ satisfies $r < pq -p -q$, then $\pi_1(S^3_{r} (C))$ is left-orderable.
\end{theorem}
This result is a special case of a more general observation pertaining to satellite knots that is discussed in Section \ref{sec:satellite}.  Notice that Theorem 6 makes no reference to the original knot being $r$-decayed. However, restricted to $r$-decayed knots, Theorem \ref{thm:cable} and Theorem \ref{thm:locable} combine to produce an interval of surgery coefficients for which the left-orderability of the associated quotient is not determined. More precisely:
\begin{question} If $K$ is $r$-decayed and $C$ is a $(p,q)$-cable of $K$ with $\frac{q}{p}>r$, can Theorem  \ref{thm:cable} and Theorem \ref{thm:locable} be sharpened to determine when $\pi_1(S^3_{r'} (C))$ is left-orderable for $r'$ satisfying $pq -p -q < r'  \leq  pq$?\end{question}  

\subsection*{Acknowledgments} We thank Josh Greene, Tye Lidman and Dale Rolfsen for helpful comments on an earlier draft  of this paper.

%% file: section/reformulation2.tex



We begin with a reformulation of $r$-decay that will be essential in connecting this work with the results of \cite{CW2010}. This will require the following lemma:

\begin{lemma}  
\label{lem:iff}
Let $G$ be a left-orderable group containing elements $g,h$.  If $g\in\sP$ implies $h\in\sP$ for every positive cone $\sP$, then $g\in\sP$ if and only if $h\in\sP$.
\end{lemma} 
\begin{proof}
We need only show the converse, namely $h\in\sP$ implies $g\in\sP$ for every positive cone $\sP \subset G$.  For a contradiction, suppose this is not the case, so there exists a positive cone such that $h\in\sP$ and $g\notin\sP$.  Consider the positive cone $\sQ=\sP^{-1}$, defining the reverse ordering of $G$.  This gives $g\in\sQ$ and $h\notin\sQ$, contradicting our assumption.
\end{proof}

\begin{proposition}
\label{prp:equiv}
 A knot $K$ is $r$-decayed if and only if for every positive cone $\sP \subset \pi_1(K)$ there exists a strictly increasing sequence of positive rational numbers $\{r_i\}$ with $r_i\to\infty$ satisfying \begin{itemize}
\item[(1)] $r = r_0$, and 
\item[(2)] $\alpha_r\in\sP$ implies $\alpha_{r_i}\in\sP$ for all $i$.
\end{itemize}
\end{proposition}

\begin{proof}
Suppose that $K$ is $r$-decayed , and let $\sP$ be any positive cone.  Choose a strictly increasing sequence of rational numbers $\{ r_i \}$ with $r_0 = r$ and $r_i\to\infty$.  Whenever $\alpha_{r}=\alpha_{r_0}  \in \sP$ we have $\sS_r \cap \sP \neq \emptyset$, so that $\sS_r \cap \sP = \sS_r$ since $K$ is $r$-decayed.  It follows that $\alpha_{r_i}\in \sS_r \subset \sP$ for all $i$.

To prove the converse, let $\sP$ be a positive cone for $\pi_1(K)$. Fix a strictly increasing sequence $\{r_i\}$ of rational numbers limiting to infinity and satisfying (1) and (2). Suppose that $\alpha_r \in \sP$, then by assumption $\alpha_{r_i}\in\sP$ for all $i>0$.

Now suppose that $\mu^m \lambda^n$ is an element of $\sS_r $.  Choose $r_{i}, r_{i+1}$ with corresponding reduced forms $\frac{p_i}{q_i},  \frac{p_{i+1}}{q_{i+1}}$ such that $r_i < \frac{m}{n} < r_{i+1}$.  By solving
\begin{align*}
q_i a  +   q_{i+1} b & =  cn \\ 
p_i a  +   p_{i+1} b & = cm
\end{align*}
 we can find positive integers $a, b$ and $c$ such that 
$(\mu^{p_i} \lambda^{q_i})^a (\mu^{p_{i+1}} \lambda^{q_{i+1}})^b = (\mu^m \lambda^n)^c$.
Explicitly, Cramer's rule gives
\[
a=\left|\begin{matrix} n & q_{i+1} \\ m & p_{i+1}
\end{matrix}\right|, \quad
 b =  \left| \begin{matrix} q_i & n \\ p_i & m \end{matrix} \right|,  \quad
c = \left|\begin{matrix} q_i & q_{i+1} \\ p_i & p_{i+1} \end{matrix} \right|;
\]
note that all these quantities are positive because of our restriction  $r_i < \frac{m}{n} < r_{i+1}$  (compare \cite[Lemma 17]{CW2010}) .
This shows that $\mu^m \lambda^n$ is positive, since its $c$-th power is expressed as a product of positive elements.  Hence $\sS_r \cap \sP = \sS_r$.

This establishes the implication $\alpha_{r} \in \sP \Rightarrow \sS_r \subset \sP$ for every positive cone $\sP$.  By Lemma \ref{lem:iff}, this is equivalent to $\sS_r \cap \sP = \sS_r$ or $\sS_r \cap \sP = \emptyset$ for every positive cone $\sP$, so that $K$ is $r$-decayed.
\end{proof} 

\begin{remark}In practice, it is often more natural to establish $\alpha_r\in\sP$ implies $\alpha_{r_i}^{w_i}\in\sP$ for all $i$, where $w_i\in\bN$ (see in particular  the proofs of Lemma \ref{lem:lemma1} and Lemma \ref{lem:lemma2}).   This situation arises when one constructs (for a given positive cone $\sP$) a sequence of \textit{unreduced} rationals $\{ r_i \} = \{ \frac{p_i}{q_i} \}$ for which $\gcd( p_i, q_i ) = w_i \geq 1$, and $\mu^{p_0} \lambda^{q_0} \in \sP$ implies $\mu^{p_i} \lambda^{q_i} \in \sP$ for all $i$.
 Notice that the implication $\alpha_r\in\sP$ implies $\alpha_{r_i}^{w_i}\in\sP$ still allows us to apply Proposition \ref{prp:equiv}, since $\alpha_{r_i}^{w_i}\in\sP$ if and only if $\alpha_{r_i}\in\sP$ (this simple observation holds in any left-orderable group). Ultimately, this results in more flexibility in selecting the sequence $\{r_i\}$.\end{remark}

The equivalence established in Proposition \ref{prp:equiv} shows that all examples considered in \cite{CW2010} are $r$-decayed for certain $r$, as \cite[Corollary 11]{CW2010} is a special case of Proposition \ref{prp:equiv}.  

\begin{theorem}\cite[Theorem 24, Theorem 28 and Theorem 30]{CW2010}
\label{thm:decayedknots}
\begin{enumerate}
\item The $(p,q)$-torus knot is $(pq-1)$-decayed for all positive, relatively prime pairs of integers $p,q$.
\item The $(-2, 3, q)$-pretzel knot is $(10+q)$-decayed for all odd $q \geq 5$.
\item The $(3,q)$-torus knot with one positive full twist added along two strands is $(3q+2)$-decayed, for all positive $q$ congruent to $2$ modulo $3$.
\end{enumerate}
\end{theorem}
\begin{proof}
We consider the case of $K_q$, the $(-2, 3, q)$-pretzel knot with $q \geq 5$ odd, the other cases are similar.   Set $r = 10+q$, and $r_i = r+i$.  It is shown in \cite{CW2010} that for every positive cone $\sP$ in $\pi_1(K_q)$, the implication $\alpha_{r} \in  \sP \Rightarrow \alpha_{r_i} \in \sP$ holds for all $i \geq 0$.  This means that for every left-ordering of $\pi_1(K_q)$,  the integer sequence $\{r_i \}$ satisfies the properties required by Proposition \ref{prp:equiv}, and we conclude that $K_q$ is $r$-decayed. 
\end{proof}

Note that the above proof illustrates some particularly special behaviour, as the rational sequences $\{r_i \}$ required by Proposition \ref{prp:equiv} (which {\em a priori} may be different for each left-ordering) are replaced by a single integer sequence sufficient for every left-ordering. Thus, Proposition \ref{prp:equiv} provides a more workable method (than used previously) for checking when a knot has surgeries that yield a non-left-orderable fundamental group.  Combined with the material established in \cite[Section 2]{CW2010}, we provide a short proof of Theorem \ref{thm:decay}.
\begin{proof}[Proof of Theorem \ref{thm:decay}]
For contradiction, assume that $\pi_1(S^3_{r'}(K))$ is left-orderable for some $r' \geq r$, and consider the short exact sequence
\[
1 \rightarrow \langle \langle \alpha_{r'} \rangle\rangle \stackrel{i}{\rightarrow} \pi_1(K) \stackrel{f}{\rightarrow} \pi_1(S^3_{r'}(K)) \rightarrow 1,
\]
as defined in the introduction.   Let $\mu, \lambda \in \pi_1(K)$ denote the meridian and longitude.
Since $\pi_1(S^3_{r'}(K))$ is left-orderable, $\langle \langle \alpha_{r'} \rangle \rangle \cap \langle \mu, \lambda \rangle = \langle \alpha_{r'} \rangle $ (see proof of \cite[Proposition 20]{CW2010}).  In particular, if we fix an arbitrary rational number $s_0 > r'$, then $ f(\alpha_{s_0} ) \neq 1$.  Thus, we may choose a positive cone $\sQ$ in $\pi_1(S^3_{r'}(K))$ that contains $f(\alpha_{s_0})$.  Next, choose a positive cone $\sQ' \subset \langle \langle \alpha_{r'} \rangle \rangle$ not containing $\alpha_{r'}$, and define a positive cone $\sP \subset \pi_1(K)$ by
\[ \sP = i(\sQ') \sqcup f^{-1}(\sQ).
\]
Note that $\alpha_{r'} \notin \sP$, and $ \alpha_{s_0} \in \sP$.  

This is a standard construction for creating a left-ordering of a group using a short exact sequence, here the result is a left-ordering of $\pi_1(K)$ with positive cone $\sP$, relative to which the subgroup $\langle \langle \alpha_{r'} \rangle\rangle$ is convex.   Because $\langle \langle \alpha_{r'} \rangle\rangle$ is convex, the intersection  $\langle \langle \alpha_{r'} \rangle \rangle \cap \langle \mu, \lambda \rangle = \langle \alpha_{r'} \rangle $ is convex in the restriction ordering of $\langle \mu, \lambda \rangle$.  Therefore, \cite[Proposition 18]{CW2010} shows that all slopes $\alpha_{s}$ with $s >r'$ must have the same sign.  In particular, since $\alpha_{s_0}$ is positive it follows that all slopes $\alpha_{s}$ with $s >r'$ are positive, so that
\[ Q \cap \sS_r = \{ \alpha_{s} | s > r' \}.
\]
Therefore, $K$ is not $r$-decayed.
\end{proof}

We remark that there is a more geometric argument establishing Theorem \ref{thm:decay}, that relies upon an understanding of the topology of the space of left-orderings of $\mathbb{Z} \oplus \mathbb{Z}$ (see \cite[Section 3]{Sikora2004} and \cite[Chapter 6]{Clay2010}).  Roughly, every left-ordering of the knot group $\pi_1(K)$ restricts to a left-ordering of the peripheral subgroup that defines a line in $\bZ \oplus \bZ$, with all positive elements of $\bZ \oplus \bZ$ on one side of the line, and all the negative elements on the other side.  As a result, given two rationals $r_1 < r_2$ corresponding to slopes $\alpha_{r_1}$ and  $\alpha_{r_2}$ that have the same sign in every left-ordering,  no left-ordering can restrict to an ordering of the peripheral subgroup with corresponding slope $s$ between $r_1$ and $r_2$.  The proof of Theorem \ref{thm:decay} then follows from checking that whenever $\pi_1(S^3_{r'}(K))$ is left-orderable, we can define a left-ordering of $\pi_1(K)$ that restricts to yield a line of slope $r'$ in the peripheral subgroup (compare \cite[Proof of Theorem 9]{CW2010}).




%% file: section/proof-option.tex

We recall the construction of a cabled knot in order to fix notation. Consider the $(p,q)$-torus knot $T_{p,q}$, where $p, q >0$ are relatively prime. As the closure of a $p$-strand braid, this knot may be naturally viewed in a solid torus $T$ by removing a tubular neighbourhood of the braid axis. The complement of $T_{p,q}$ in $T$ is referred to as a $(p,q)$-cable space. Now given any knot $K$ in $S^3$, the cable knot $C_{p,q}(K)$ is obtained by identifying  the boundary of $T$ with the boundary of $S^3\smallsetminus\nu(K)$, identifying the longitude of $T$ with the longitude $\lambda$ of $K$. We will denote this cable knot by $C$ whenever this simplified notation does not cause confusion.

The knot group $\pi_1(C)$ may be calculate via the Seifert-Van Kampen Theorem, by viewing the complement $S^3\smallsetminus\nu(C)$ as the identification of the boundaries of $S^3\smallsetminus\nu(K)$ and a solid torus $D^2\times S^1$ along an essential annulus with core curve given by the slope $\mu^q\lambda^p$. If $\pi_1(D^2\times S^1)=\langle t \rangle$ then this gives rise to a natural amalgamated product
\[\pi_1(C)\cong\pi_1(K)\ast_{\mu^q\lambda^p=t^p}\bZ.\]

Consulting \cite[Section 3]{CW2010}, the meridian and longitude for $C$ may be calculated as \[\mC= \mu^u\lambda^vt^{-v} \quad {\rm and} \quad \lC= \mC^{-pq}t^p\] where $u$ and $v$ are positive integers satisfying $pu-qv=1$ (compare \cite[Proof of Theorem 3.1]{Tsau1988}). 

Suppose that the knot $K$ is $r$-decayed, and choose cabling coefficients $p$ and $q$ so that $q/p >r$.  To begin, we choose a positive cone $\sP\subset\pi_1(C)$ and assume that $\mu_C^{pq} \lambda  = t^p$ is positive.  This means that $t^p  = \mu^q \lambda^p\in\sP$, so every element $\mu^m \lambda^n$ is positive whenever $m/n >r$, since $K$ is $r$-decayed.  

 Our method of proof will be to check that the cable is $pq$-decayed by using the equivalence from Proposition \ref{prp:equiv}. In particular, we will show that for the given positive cone $\sP \subset \pi_1(C)$ there exists an unbounded sequence of increasing  rationals $\{r_i \}$ with $r_0 = pq$, such that our assumption $\alpha_{pq} = \mu_C^{pq} \lambda_C\in\sP$ implies $\alpha_{r_i} \in \sP$ for all $i>0$.

First consider the case when $\mu_C $ is positive in the left-ordering defined by $\sP$. Here, $\mu_C^{pq+N} \lambda_C $ is positive for $N \geq 0$,  as it is a product of positive elements. Therefore in this case it suffices to choose $r_i = pq+i$ for all $i \geq 0$.

For the remainder of the proof, we assume that $\mu_C$ is negative.  For repeated use below, we also observe the crucial identity
\[ (t^{-v})^p (\mu^u \lambda^v)^p = (t^p)^{-v} \mu^{up} \lambda^{vp} = \mu^{-qv} \lambda^{-pv} \mu^{up} \lambda^{vp} = \mu^{pv -qu} = \mu,
\]
and recall that $t^p$ commutes with $\mu, \lambda, \mu_C$, and $\lambda_C$.   Therefore, we also have
\[ (\mu^u \lambda^v)^p (t^{-v})^p= (t^{-v})^p (\mu^u \lambda^v)^p  = \mu.
\]

Let $k$ be an arbitrary non-negative integer, and consider the element
\[ \mu^{-k} (t^{-v} \mu^u \lambda^v) \mu^k.
\] 
If this element is positive for some $k$, then the required sequence is provided by Lemma \ref{lem:lemma1} (proved below).  Therefore, we may assume that 
\begin{equation}\label{option1} 
\mu^{-k} (t^{-v} \mu^u \lambda^v) \mu^k\notin\sP
\end{equation}
for all $k$.

Similarly, for $k$ a non-negative integer, we consider
\[ (\mu^{-k} t^{-v} \mu^k)^{p-1} (\mu^u \lambda^v)^{p-1}.
\]
If this element is positive for some non-negative $k$, then we can create the required sequence using Lemma \ref{lem:lemma2} (proved below).
Therefore, we may assume that 
\begin{equation}\label{option2} (\mu^{-k} t^{-v} \mu^k)^{p-1} (\mu^u \lambda^v)^{p-1} \notin\sP
\end{equation}
for all $k$.

Observe that
\[ (\mu^{-k} t^{-v} \mu^k)^{p-1} (\mu^u \lambda^v)^{p-1}  = (\mu^{-k} t^{v} \mu^k) (\mu^{-k} t^{-vp} \mu^k)  (\mu^{up-u} \lambda^{vp-v}), 
\]
which, recalling that $t^p$ commutes with the elements $\mu$ and $\lambda$, simplifies to give
\[ (\mu^{-k} t^{v} \mu^k) (\mu^{-u} \lambda^{-v}) t^{-vp} \mu^{up} \lambda^{vp} =(\mu^{-k} t^{v} \mu^k) (\mu^{-u} \lambda^{-v}) \mu = \mu^{-k} t^{v} \lambda^{-v} \mu^{-u} \mu^{k+1} \notin\sP
\]
for all $k$. Taking inverses yields
\[ \mu^{-k-1} \mu^u \lambda^v t^{-v} \mu^k = \mu^{-k-1} \mu_C \mu^k \in\sP.
\]



For the following lemmas, let $>$ denote the left-ordering defined by the positive cone $\sP$, so that $h > g$ whenever $g^{-1} h \in \sP$.   We can then calculate:

\begin{lemma} If (\ref{option1}) and (\ref{option2}) hold for all $k \geq 0$, then $\mu^{N+q} \lambda^p$ must be positive for all $N \geq 0$.\end{lemma}
\begin{proof} 

Since $ \mu^{-k-1} \mu_C \mu^k >1$, left-multiplying by $\mu^{k+1}$ gives $\mu_C \mu^k > \mu^{k+1}$ for all $k \geq 0$.  Setting $k=0$ we obtain $\mu_C > \mu$, so that left-multiplying by $\mu_C$ gives rise to
\[ \mu_C^2 > \mu_C \mu.\]
By setting $k=1$, we get
\[\mu_C \mu > \mu^2,\]
which combines with the previous expression to give $\mu_C^2 > \mu^2$.   Continuing in this manner, we obtain $\mu_C^N > \mu^N$ for all $N \geq 0$.
Left-multiplying by $t^p$, it follows that
\[ \mu_C^{N+pq} \lambda_C = \mu_C^N t^p > \mu^N t^p = \mu^{N+q} \lambda^p >1,
\]
where the final inequality follows from the fact that $(N+q)/p >q/p>r$ and $K$ is $r$-decayed. \end{proof}

As in the first case, we may now choose the sequence of rationals $r_i = pq+i$ for all $i \geq 0$, and the requirements of Proposition \ref{prp:equiv} are met.

To conclude the proof, we establish Lemma \ref{lem:lemma1} and Lemma \ref{lem:lemma2}. 

\begin{lemma}
\label{lem:lemma1}
If $\mu^{-k} (t^{-v} \mu^u \lambda^v) \mu^k \in \sP$ for some $k \geq 0$, then there exists a sequence of rationals $\{ r_i \}$ such that $\alpha_{r_i} >1$ for all $i$.
\end{lemma}
\begin{proof}
For $N \geq 0 $, we rewrite $\mu_C^N$ as
\[ \mu_C^N = \mu^u \lambda^v \mu^k (\mu^{-k} (t^{-v} \mu^u \lambda^v)^N \mu^k) \mu^{-u-k} \lambda^{-v}.
\]
Fix a positive integer $s$ that is large enough so that $(sq-u-k)/(sp-v) > r$, this is possible because $q/p>r$.   Next, the product $\mu_C^{N+pqs} \lambda_C^s = \mu_C^N t^{ps}$ becomes $\mu_C^N \mu^{sq} \lambda^{sp}$, which is equal to
\[ \mu^{u+k} \lambda^v (\mu^{-k} (t^{-v} \mu^u \lambda^v)^N \mu^k) (\mu^{qs-u-k} \lambda^{ps-v}).
\]
This is a product of positive elements, because:
\begin{enumerate}
\item $\mu^{u+k} \lambda^v >1$ because $(u+k)/v > q/p>r$, and 
\item $\mu^{qs-u-k} \lambda^{ps-v} >1$ because $(sq-u-k)/(sp-v) > r$,
\end{enumerate}
while the quantity $\mu^{-k} (t^{-v} \mu^u \lambda^v)^N \mu^k$ is positive by assumption.
Therefore, in this case we choose our sequence of rationals to be \[r_i = \frac{pqs+i}{s}\] for $i \geq 0$, this guarantees that the associated slopes $\alpha_{r_i}$ are positive in the given left-ordering.
\end{proof}

\begin{lemma}
\label{lem:lemma2}
If $\mu^{-k} (t^{-v} \mu^u \lambda^v) \mu^k \notin \sP$ for all $k \geq 0$, and $ (\mu^{-k} t^{-v} \mu^k)^{p-1} (\mu^u \lambda^v)^{p-1}  \in \sP$ for some $k \geq 0$, then there exists a sequence of rationals $\{ r_i \}$ such that $\alpha_{r_i} >1$ for all $i$.
\end{lemma}
\begin{proof}
Fix $k \geq 0$ such that $\mu^{-k} (t^{-v} \mu^u \lambda^v) \mu^k<1$ and $ (\mu^{-k} t^{-v} \mu^k)^{p-1} (\mu^u \lambda^v)^{p-1} >1$, and
let $n$ be the smallest positive integer such that 
\[ (\mu^{-k} t^{-v} \mu^k)^{n} (\mu^u \lambda^v)^{n} >1,
\]
and
\[ (\mu^{-k} t^{-v} \mu^k)^{n-1} (\mu^u \lambda^v)^{n-1} <1,
\]
note that $1<n\leq q-1$.
Note that we may rearrange these two expressions, so that
\[ \mu^{-k} t^{-vn}  (\mu^u \lambda^v)^{n} \mu^k>1,
\]
and 
\[ \mu^{-k}  (\mu^u \lambda^v)^{1-n}  t^{-v(1-n)} \mu^k >1.
\] 
Then we can rewrite $\mu_C^N$ for $N \geq 1$ as follows:
\[ \mu_C^N = \mu^{u+k} \lambda^v (\mu^{-k} t^{-v(1-n)} \mu^k)[ (\mu^{-k} t^{-vn}  (\mu^u \lambda^v)^{n} \mu^k) \mu^{-k}  (\mu^u \lambda^v)^{1-n}  t^{-v(1-n)} \mu^k ]^{N-1} \mu^{-k} t^{-vn}.
\]
In the above expression, the quantity inside the square brackets is a product of positive elements.   Denote this quantity by P.  Choose an integer $s$ such that $(qs-k)/ps > r$.  Then considering the slope $\mu_C^{N+pq(v+s)} \lambda_C^{v+s} = \mu_C^N t^{p(v+s)}$, we find
\[ \mu_C^N t^{p(v+s)} = \mu^{u+k} \lambda^v (\mu^{-k} t^{-v(1-n)} \mu^k)P^{N-1} \mu^{qs-k}\lambda^{ps} t^{pv-vn} .
\]
This is a product of positive elements, because:
\begin{enumerate}
\item $\mu^{u+k} \lambda^v>1$, since $(u+k)/v > q/p > r$.
\item $\mu^{-k} t^{-v(1-n)}  \mu^k>1$, because if we consider its $p$-th power, we can use the fact that $t^p$ commutes with all peripheral elements so that
\[ (\mu^{-k} t^{-v(1-n)}  \mu^k)^p = t^{-pv(1-n)} >1.
\]
The final inequality follows from $-pv(1-n)>0$.
\item $\mu^{qs-k}\lambda^{ps} >1$, because $s$ is chosen so that $(qs-k)/ps > r$.  
\item $t^{pv-vn}>1$, because $pv-vn>0$.
\end{enumerate}
Therefore, in this case we may choose our sequence of rationals to be \[ r_i = \frac{i+pq(v+s)}{v+s} \] for $i \geq 0$,  as the corresponding elements $\mu_C^{i+pq(v+s)} \lambda_C^{v+s} $ are positive in the left-ordering for $i \geq 0$.
\end{proof}

%% file: section/satellite-alt.tex

Let $T$ denote the solid torus containing a knot $K^P$, we require that $K^P$ is not contained in any $3$-ball inside $T$.  The knot $K^P$ will be called the pattern knot.    Let $K^C$ denote a knot in $S^3$, $K^C$ will be called the companion knot.   We construct the satellite knot $K$ with pattern $K^P$ and companion $K^C$ as follows.   

Let $h : \partial T \rightarrow \partial ( S^3 \smallsetminus \nu(K^C))$ denote a diffeomorphism from the boundary of $T$ to the boundary of the complement of $\nu(K^C)$, which carries the longitude of $\partial T$ onto the longitude of the knot $K^C$.  The knot $K$ is then realized as the image of the knot $K^P$ in the manifold
\[  S^3 \smallsetminus \nu(K^C) \sqcup_{h} T =S^3.
\]

\begin{lemma} \cite[Proposition 3.4]{SW06}
\label{lem:homo}
There exists a homomorphism $\phi: \pi_1(K) \rightarrow \pi_1( K^P)$
that perserves peripheral structure.
\end{lemma}
\begin{proof}
We can compute the fundamental group $\pi_1(K )$ by using the Seifert-Van Kampen theorem.   Since 
\[ S^3 \smallsetminus \nu(K) = S^3 \smallsetminus \nu(K^C) \sqcup_{h} T \smallsetminus \nu (K^P) ,
\]
the group $\pi_1(K)$ is the free product $\pi_1(K^C) * 
\pi_1( T \smallsetminus \nu(K^P) ) $, with amalgamation as follows:  The meridian of $K^C$ is identified with the meridian of $T$, and the longitude of $K^C$ is identified with the longitude of $T$.

Let $N$ denote the normal closure in $\pi_1(K)$ of the commutator subgroup of $\pi_1(K^C)$.    The quotient $\pi_1(K) /N$ can be considered as the result of killing the longitude of $T$.  Topologically we can think of this quotient as gluing a second solid torus $T'$ to the torus $T$ containing $K^P$, in such a way that the meridian of $T'$ is glued to the longitude of $T$.  The result is that $\pi_1(K^C)$ collapses to a single infinite cyclic subgroup, and the group $\pi_1(K) /N$ is isomorphic to $\pi_1(K^P)$.  The desired homomorphism $\phi$ is the quotient map $\pi_1(K ) \rightarrow \pi_1(K) /N$. 
\end{proof}

\begin{proposition}
\label{prop:satellitelo}
Suppose that $K$ is a satellite knot with pattern knot $K^P$, and $r \in \mathbb{Q}$ is any rational number.  If $\pi_1(S^3_{r}(K^P))$ is left-orderable and $S^3_{r}(K)$ is irreducible, then $\pi_1(S^3_{r}(K))$ is left-orderable.
\end{proposition}
\begin{proof}
By Lemma \ref{lem:homo}, there exists a homomorphism $\phi \co \pi_1(K) \rightarrow \pi_1( K^P)$ that preserves peripheral structure,  so there exists an induced map
\[ \phi_{r} : \pi_1(S^3_{r} (K) ) \rightarrow \pi_1(S^3_{r} (K^P) )
\]
for every $r \in \mathbb{Q}$.
Whenever $\pi_1(S^3_{r} (K^P) )$ is left-orderable the image of $\phi_{r}$ is nontrivial and $\pi_1(S^3_{r} (K) )$ is left-orderable \cite[Theorem 1.1]{BRW2005} .
\end{proof}

\begin{proof}[Proof of Theorem \ref{thm:locable}]
By \cite{Gordon83}, $pq$-surgery on a $(p,q)$-cable knot yields a reducible manifold.   Since the minimal geometric intersection number between reducible slopes is $\pm 1$ \cite{GL96}, $r$-surgery on a $(p,q)$-cable yields an irreducible manifold whenever $r<pq-p-q$.  Moreover, a $(p,q)$-cable knot can be described as a satellite knot with pattern knot $T_{p,q}$, the $(p,q)$-torus knot.  Therefore, for $r<pq-p-q$ we can apply Proposition \ref{prop:satellitelo} to conclude that $\pi_1(S^3_{r}(K))$ will be left-orderable whenever $\pi_1(S^3_r (T_{p,q}))$ is left-orderable.  

We may now combine known results for surgery on torus knots in this setting. On the one hand, $\pi_1(S^3_r (T_{p,q}))$ is an L-space whenever $r\ge 2g-1$ \cite[Proposition 9.5]{OSz-rational} (see in particular \cite[Lemma 2.13]{Hedden2009}), where $g=g(T_{p,q})$ is the Seifert genus given by $g(T_{p,q})= \frac{1}{2}(p-1)(q-1)$. On the other, since $S^3_r (T_{p,q})$ is Seifert fibred or a connect sum of lens spaces for every $r$ \cite{Moser1971},  $S^3_r (T_{p,q})$ is an L-space if and only if $\pi_1(S^3_r (T_{p,q}))$ is not left-orderable \cite{BGW2010} (see also \cite{Peters2009,Watson2009}). In particular,  $\pi_1(S^3_r (T_{p,q}))$ is left-orderable whenever $r$ is less than $2g(T_{p,q})-1$ and the result follows.
\end{proof}